\documentclass[reqno,11pt]{amsart}
\usepackage[utf8]{inputenc}
\usepackage{amssymb, mathrsfs, amsfonts, amsmath}

\usepackage{tikz}
\usepackage[english]{babel}
\usepackage{xcolor} 
\usepackage{setspace}
\usepackage[left=3.2cm, right=3.2cm, bottom=4cm]{geometry}

\newcommand{\dx}{\, \mathrm{d}x}

\newcommand{\R}{\mathbb{R}}
\newcommand{\Z}{\mathbb{Z}}
\newcommand{\N}{\mathbb{N}}
\newcommand{\C}{\mathbb{C}}
\newcommand{\mmd}{\mathrm{d}}

\newcommand{\sinc}{\text{sinc}}

\newtheorem{theorem}{Theorem}
\newtheorem{rmk}{Remark}
\newtheorem{lemma}[theorem]{Lemma}

\newtheorem{corollary}[theorem]{Corollary}
 
\newtheorem*{conjecture}{Conjecture}
\newtheorem{question}{Question}

\def\Xint#1{\mathchoice
{\XXint\displaystyle\textstyle{#1}}%
{\XXint\textstyle\scriptstyle{#1}}%
{\XXint\scriptstyle\scriptscriptstyle{#1}}%
{\XXint\scriptscriptstyle\scriptscriptstyle{#1}}%
\!\int}
\def\XXint#1#2#3{{\setbox0=\hbox{$#1{#2#3}{\int}$}
\vcenter{\hbox{$#2#3$}}\kern-.5\wd0}}

\def\dashint{\Xint-}

\title[Fourier uniqueness for powers]{Fourier uniqueness pairs of powers of integers}
\author[J. P. G. Ramos]{Jo\~ao P. G. Ramos}
\author[M. Sousa]{Mateus Sousa}

\address{João Pedro Ramos, Mathematical Institute, Universität Bonn \\
Endenicher Allee 60, 53115, Bonn, Germany.}
\email{jpgramos@math.uni-bonn.de}
	
\address{Mateus Sousa, Ludwig-Maximilans Universität München \\
Theresienstr. 39, 80333 München, Germany.}
\email{sousa@math.lmu.de}

\begin{document}
\begin{abstract}  We prove, under certain conditions on $(\alpha,\beta)$, that each Schwartz function $f$ such that $f(\pm n^{\alpha}) = \widehat{f}(\pm n^{\beta}) = 0, \forall n \ge 0$ 
must vanish identically, complementing  a series of recent results involving uncertainty principles, such as the pointwise interpolation formulas by Radchenko and Viazovska and the Meyer--Guinnand construction of self-dual crystaline measures.
\end{abstract}
\maketitle

\section{Introduction}
Given an integrable function $f:\R\rightarrow\C$, we define its Fourier transform by 
\begin{equation}\label{eq:fourier_transform}
    \widehat{f}(\xi):=\int_{\R}f(x)e^{2\pi i x\cdot \xi}\dx. 
\end{equation}
Let us consider the following classical problem in Fourier analysis: 

\begin{question} Given a collection $\mathcal{C}$ of functions $f:\R\rightarrow\C$, what conditions can we impose on two sets $A,\widehat{A} \subset \R$ to ensure that the only function $f\in \mathcal{C}$ such that $f(x)= 0$ for every $x\in A$ and $\widehat{f}(\xi)=0$ for every $\xi\in\widehat{A}$ is the zero function?
\end{question} 
\noindent Inspired by the notion of Heisenberg uniqueness pairs introduced by Hedenmalm and Montes-Rodr\'igues in \cite{HMR}, (see also \cite{GS19,Lev}), we refer to such pair of sets $(A,\widehat{A})$  as a \emph{Fourier uniqueness pair} for $\mathcal{C}$ for a natural reason: the values of $f(x)$ for $x\in A$ and $\widehat{f}(\xi)$ for $\xi \in B$ determine at most one function $f\in \mathcal{C}$. For simplicity, when $A=\widehat{A}$, we will say that $A$ is a Fourier uniqueness set for $\mathcal{C}$.

\smallskip

Perhaps the most classical result which answers such a question is the celebrated Shannon--Whittaker interpolation formula, which states that a function $f\in L^2(\R)$ whose Fourier transform  $\widehat{f}$  is supported on the interval $[-\delta/2,\delta/2]$ is given by the formula
\begin{equation*}
    f(x)=\sum_{k=-\infty}^{\infty}f(k/\delta)\sinc(\delta x- k),
\end{equation*}
where convergence holds both in the $L^2(\R)$ sense and uniformly on the real line, and $\sinc(x)=\tfrac{\sin(\pi x)}{\pi x}$. This means that the pair $\frac{1}{\delta}\Z$ and $\R\backslash[-\delta/2,\delta/2]$ forms a Fourier uniqueness pair  for the collection $\mathcal{C}=L^2(\R)$. More recently, Radchenko and Viazovska \cite{RV} obtained a related interpolation formula for Schwartz functions: there are even functions $a_k\in\mathcal{S}(\R)$ such that, for any given even function $f:\R\rightarrow\C$ that belongs to the Schwartz class $\mathcal{S}(\R)$, one has the following identity:
\begin{equation}\label{eq:interpolation_schwartz}
    f(x)=\sum_{k=0}^{\infty}f(\sqrt{k})a_k(x)+\sum_{k=0}^{\infty}\widehat{f}(\sqrt{k})\widehat{a_k}(x),
\end{equation}
where the right-hand side converges absolutely. This interpolation result has as immediate consequence: the set $\sqrt{\Z_+}$ of square roots of non-negative integers is a Fourier uniqueness set for the collection of even\footnote{In \cite{RV}, the authors also have results for functions which are not even, but we chose to present this version to keep technicalities to a minimum.} Schwartz functions. 

\smallskip The two theorems we just presented to motivate our question are, in fact, also instances of the intimate relationship between interpolation and summation formulas. Indeed, as previously mentioned, the Shannon--Whittaker interpolation formula is directly related to the Poisson summation formula
\begin{equation*}
\sum_{m \in \Z} f(m) = \sum_{n \in \Z} \widehat{f}(n),
\end{equation*}
and the result by Radchenko and Viazovska is, in fact, a by-product of the development of several summation formulas, having relationship to modular forms and the sphere packing problem (see, for instance, \cite{CKMRV, CKMRV2, Viazovska8}). In fact, the lower bound for the Fourier analysis problem 
corresponding to the sphere packing problem (see \cite{CE03}) is directly related to the Poisson summation formula for lattices: if $\Lambda \subset \R^n$ is a lattice with fundamental region having volume 1, 
then 
\[
\sum_{\lambda \in \Lambda} f(\lambda) = \sum_{\lambda^{*} \in \Lambda^*} \widehat{f}(\lambda^*),
\]
where $\Lambda^*$ denotes the dual lattice of $\Lambda.$ Also, in \cite{CohnGoncalves}, the authors need a summation formula stemming from an Eisenstein series $E_6,$ which implies, in particular, that for each radial Schwartz function $f: \R^{12} \to \C,$ there exists constants $c_j>0$ such that 
\[
f(0) - \sum_{j \ge 1} c_j f(\sqrt{2j}) = - \widehat{f}(0) + \sum_{j \ge 1} c_j \widehat{f}(\sqrt{2j}).
\]
These concepts seem to be all tethered to the notion of \emph{crystaline measures} and self-duality, as discussed in \cite{LO13, LO15, Meyer}. A crystaline measure is essentially a tempered distribution with locally finite support whose Fourier transform has these same properties. For instance, Poisson summation implies that 
\[
\delta_{\Z} = \widehat{\delta_{\Z}},
\]
which shows that the usual delta distribution at the integers is not only a crystaline measure, but also a \emph{self-dual} one with respect to the Fourier transform. Meyer then discusses other examples of crystaline measures with certain self-duality properties, and, similarly to the strategy used by Radchenko and Viazovska, 
uses modular forms to construct explicity examples of non-zero crystaline measures $\mu$ supported in $\{\pm\sqrt{k+a}, k \in \Z\},$  for $a \in \{9,24,72\}.$ It is interesting to point out that Meyer calls out the readers attention to the highly unexplored problem of analyzing when there is a non-zero
crystaline measure $\mu$ such that both itself and its Fourier transform have support on a given locally finite set $\{\lambda_k\colon k \in \Z\}.$ 

\smallskip Back to Fourier uniqueness pairs, while both the Shannon--Whittaker and Radchenko--Viazovska results provide Fourier uniqueness pairs by means of interpolation identities, and such explicit formulas are not always available and usually depend on special properties of the sets involved, which are somewhat rigid. In the case of the Shannon--Whittaker formula, the set $\frac{1}{\delta}\Z$ plays an special role because of the Poisson summation formula. In the case of the Radchenko--Viazovska interpolation, the set $\sqrt{\Z_+}$ becomes important due to
special properties of certain modular forms involved in their proofs. Perturbing these sets breaks down the proofs of these theorems, and sometimes even the existence of such interpolation formulas. Nevertheless, the Fourier uniqueness pair property is inherently less rigid as a condition than an interpolation formula, which might lead to uniqueness results even in the absence of possible interpolation formulas.  

\smallskip
For instance, define a set $\Lambda\subset \R$ to be uniformly separated if there is a number $\delta=\delta(\Lambda)>0$ such that $|\lambda - \lambda'|>\delta$ whenever $\lambda,\lambda'\in\Lambda$ and $\lambda\neq \lambda'$. Given an uniformly separated set $\Lambda$, we define its lower density and upper density, respectively, as the numbers
\begin{align}
    \mathcal{D}^{-}(\Lambda)&=\liminf_{R\rightarrow\infty} \inf_{x\in\R}\frac{|\Lambda\cap[x-R,x+R]|}{2R} \\
    \mathcal{D}^{+}(\Lambda)&=\limsup_{R\rightarrow\infty} \sup_{x\in\R}\frac{|\Lambda\cap[x-R,x+R]|}{2R}.
\end{align}
And when these numbers coincide we call it the density of $\Lambda$. As a corollary of the work of Beurling \cite{Beurling} and Kahane \cite{Kahane} about sampling sets, any pair $\Lambda$ and $\R\backslash[-2\pi\delta,2\pi\delta]$ forms uniqueness sets for $L^2(\R)$ if $\Lambda$ is uniformly separated and $\mathcal{D}^{-}(\Lambda)>\delta$. This means: any uniformly separated set that is more dense than $\frac{1}{\delta}\Z$ produces a pair of uniqueness sets for $L^2(\R)$, and one can readily see that this condition, at least in terms of density, is essentially sharp just by analysing subsets of $\frac{1}{\delta}\Z$.   

\smallskip.

{Another instance of this density situation has to do with the aforementioned Heisenberg uniqueness pairs. In \cite{HMR}, the authors study pairs of sets $(\Gamma,\Lambda)$, where $\Gamma \subset \R^2$, which is a finite disjoint union of smooth curves, and $\Lambda \subset \R^2$, which have the following property: whenever a measure $\mu$ supported in $\Gamma$, which is absolutely continuous with respect to the arc length measure of $\Gamma$, has Fourier transform $\widehat{\mu}$ equal to zero on the set $\Lambda$, then $\mu=0$. If a pair  $(\Gamma,\Lambda)$ has this property, it is called a Heisenberg uniqueness pair. One of the main results of \cite{HMR} is the following: Let $\Gamma=\{(x,y)\in\R^2\,:\,xy=1\}$ be the hyperbola, and $\Lambda_{\alpha,\beta}$ be the lattice cross
$$(\alpha\Z\times\{0\})\cup(\{0\}\times\beta\Z),$$
where $\alpha$ and $\beta$ are positive numbers. Then $(\Gamma,\Lambda_{\alpha\beta})$ forms a Heisenberg uniqueness pair if and only if $\alpha\beta \leq 1$. This provides yet another example of the interplay between concentration and  uniqueness properties: there is a threshold of concentration one needs to ask in order to maintain the uniqueness property, and increasing the concentration does not affect the uniqueness property.
}

\smallskip
By comparing the aforementioned interpolation theorems to the considerations in \cite{Meyer} about crystaline measures, one is naturally lead towards the following modified version of Meyer's question: if a sequence is ``more concentrated than $\sqrt{\Z}$", does it define a Fourier uniqueness set? For which notion of ``more concentrated'' could such a result possibly hold? We obtain partial progress towards this problem. 
\begin{theorem}\label{th:mainthm} Let $0< \alpha, \beta <1$ and $f\in \mathcal{S}(\R)$. Then 
\begin{itemize} 
    \item[(A)] If $f(\pm \log (n+1)))=0$ and $f(\pm n^{\alpha})=0$ for every $n\in \N$,  then $f\equiv 0$.
    \item[(B)] Let $(\alpha,\beta) \in A$, where
    \begin{align*} A  = & \left\{(\alpha,\beta) \in [0,1]^2\colon \alpha+\beta<1, 
 \text{ and either }\alpha < 1 - \frac{\beta}{1-\alpha-\beta} 
 \text{ or } \beta < 1 - \frac{\alpha}{1-\alpha-\beta}\right\}.
\end{align*}
If $f(\pm n^{\alpha})=0$ and $\widehat{f}(\pm n^{\beta})=0$ for every $n\in \N$, then $f\equiv 0$.
\end{itemize}
\end{theorem}

\begin{figure}
\centering 
\includegraphics[width=0.5\textwidth]{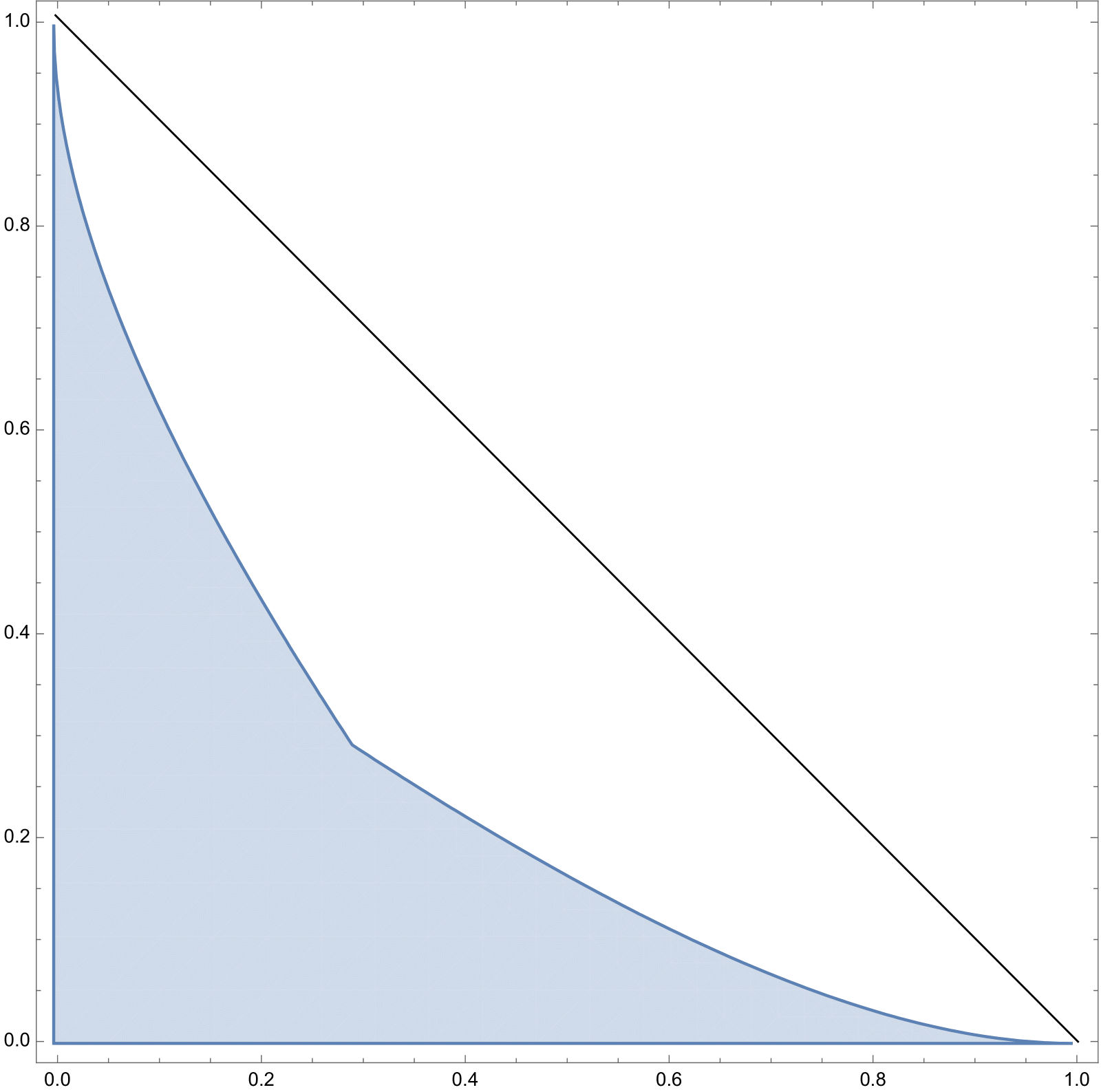}
\caption{In {\color{blue} blue}, the closure of the region $A$, with the line $\alpha + \beta =1$ in \textbf{black}.}
\end{figure}

Theorem \ref{th:mainthm} will follow by complex analytic considerations. We will prove that $f$ and $\widehat{f}$ actually have better decay than usual Schwartz functions by using the fact that the sequence of zeros of $f$ and $\widehat{f}$ grows at a certain rate, as well as the information we can obtain about the zeros of their derivatives. Once the decay is obtained, we prove either $f$ or $\widehat{f}$ admits an analytic extension of finite order, and conclude $f$ is the zero function by invoking the converse of Hadamard's theorem about growth of zeros of an entire function of finite order.
It will also become clear from the proof that the condition on the exponents $(\alpha,\beta)$ on part (ii) of theorem \ref{th:mainthm}  is a barrier of our method. We postpone a more detailed discussion about sharpness of our results to the final Section of this paper. 

\smallskip Lastly, in order to better compare our results with the ones in \cite{Meyer} and \cite{RV} we state the diagonal case of Theorem \ref{th:mainthm}.

\begin{corollary}\label{th:diagonal} Let $\alpha < 1 - \frac{\sqrt{2}}{2}.$ Then, if $f \in \mathcal{S}(\R)$ is such that $f(\pm n^{\alpha}) = \widehat{f}(\pm n^{\alpha}) = 0$ for each $n \in \mathbb{N},$ it follows that $f \equiv 0.$ 
\end{corollary}

\subsection{Organisation and notation} This article is organised as follows. In Section \ref{sec:prelim}, we mention a couple of basic ideas associating the denseness of zeros of a function and its pointwise decay. In Section \ref{sec:proof_i}, we prove the first assertion in Theorem \ref{th:mainthm}, and in Section \ref{sec:proof_ii} we work upon the ideas in the previous Section to 
prove the second part of Theorem \ref{th:mainthm}. Finally, in Section \ref{sec:comments} we make remarks, mention some corollaries of our methods and state conjectures based on the proofs presented. 

\smallskip 
Throughout this manuscript, we will use Vinogradov's modified notation $A \lesssim B$ or $A=O(B)$ to denote the existence of an absolute constant $C>0$ such that $A \le C \cdot B.$ If we allow $C$ to explicitly depend upon a parameter $\tau,$ we will write 
$A \lesssim_{\tau} B.$ In general, $C$ will denote an absolute constant that may change from line to line or from paragraph to paragraph in the argument. Finally, we adopt \eqref{eq:fourier_transform} as our normalisation for the Fourier transform. 

\section{Preliminaries}\label{sec:prelim}

\subsection{Zeros of Schwartz functions and decay} We begin by pointing out a few basic calculus facts.   \label{sec:zero_decay}

\begin{itemize}
    \item[(I)] First of all, by the mean value theorem, between two zeros of the $k$-th derivative of a function, there is a zero of the $(k+1)$-th  derivative. This means as long as there is a sequence of zeros of $f$ that converge to infinity, by a simple induction argument, there is a sequence $\{a_m^{(k)}\}_{m\in \N}$ such that
    \begin{itemize}
        \item[(I.i)] $0<a_m^{(k)}< a_{m+1}^{(k)}$ and $$\lim_{m\rightarrow+\infty}a_m^{(k)}=+\infty.$$
        \item[(I.ii)] $f^{(k)}(a_m^{(k)})=0$, and for every $m\in \Z$.
        \item[(I.iii)] For all $m\in \N,$ it holds that $[a_{m+1}^{(k)},a_{m}^{(k)}]$ is contained in the interval $[a_m,a_{m+k+1}]$, where $\{a_n\}$ are the zeros of the function $f.$ This, in particular, implies 
    $$|a_{m+1}^{(k)}-a_{m}^{(k)}|\leq |a_{m+k+1}-a_m|.$$
    \end{itemize}
\item[(II)] One can built an analogous sequence with negative zeros of the $k$-th derivative of $f$. Of course, the same can be done for $[\widehat{f}]^{(k)}$. 
\end{itemize}
Given a function $g\in \mathcal{S}(\R)$, we will use the following notation
\begin{align*}
    I_k(g)=\int_{\R} |g(y)| |y|^k \, \mmd y.
\end{align*}
The integrals $I_k(f)$ and $I_k(\widehat{f})$ will play an important role because of the following observation: whenever a point $x$ lies in an interval of the form $[a_{m+1}^{(k)},a_{m}^{(k)}]$, Fourier inversion implies
\begin{align} \label{eq:Fourier_inverstion_step}
\begin{split}
    |f^{(k)}(x)| &= |f^{(k)}(x)-f(a_{m}^{(k)})| \\
 &=  \left|\int_{\R} \widehat{f}(y)(2\pi i y)^{k} [e^{2 \pi i y x}- e^{2 \pi i y a_{m}^{(k)}}] \, \mmd y \right| \cr
 & \le (2\pi)^{k+1} I_{k+1}(\widehat{f})|x -a_{m}^{(k)}| \\& \le (2\pi)^{k}I_{k+1}(\widehat{f})|a_{m+1}^{(k)}-a_{m}^{(k)}| .
\end{split}
\end{align}
This means that the rate at which the zeros of the derivatives accumulate at infinity provides extra decay for each derivative itself. We will use this observation iteratively to improve decay bounds on our functions. 

\subsection{Fourier transforms of functions with strong decay} In addition to connecting location of zeros to decay of functions, we need to connect decay of a function to properties of its Fourier transform. The next Lemma is going to be of crucial 
importance for us throughout the proof. 

\begin{lemma}\label{th:analytic_continuation} Let $f \in \mathbb{S}(\R)$ be such that there exist two constants $C>0, \, A > 1$ for which $|f(x)| \lesssim e^{-C|x|^A}, \, \forall x \in \R.$ Then its Fourier transform 
$\widehat{f}$ can be extended to the whole complex plane as an analytic function with order at most $\frac{A}{A-1}.$ That is, for all $\varepsilon >0,$ 
$$ |\widehat{f}(z)| \lesssim_{\varepsilon} e^{|z|^{\frac{A}{A-1} + \varepsilon}}.$$ 
\end{lemma}

\begin{proof} Let $z = \xi + i \eta \in \C.$ Without loss of generality, in what follows we assume that $\text{Re}(\eta) < 0.$ We simply write
\[
\widehat{f}(z) = \int_{\R} e^{2 \pi i z \cdot x} f(x) \, \mmd x.
\]
By the decay property of $f$, it is easy to see that this integral is well-defined for each $z \in \C,$ and Morera's theorem tells us that this extension is, in fact, entire. For the assertion about its order, we have the trivial bound 
\[
|\widehat{f}(z)| \le \int_{\R} e^{-2 \pi \eta x} e^{-C|x|^A} \, \mmd x.
\]
In order to prove that the expression on the right hand side above is $\lesssim_{\varepsilon} e^{|z|^{\frac{A}{A-1} + \varepsilon}},$ we split the real line as 
$$ \R = A_{\eta} \cup B_{\eta} \cup C_{\eta},$$ 
where 
\begin{align*}
A_{\eta} = \left\{ x \in \R \colon \left|x-\left(\frac{2\pi |\eta|}{C A}\right)^{1/(A-1)} \right| \le K_A \left(\frac{2\pi |\eta|}{C A}\right)^{1/(A-1)}\right\}, \cr 
B_{\eta} = \left\{ x \in \R \colon x > (K_A + 1) \left(\frac{2\pi |\eta|}{C A}\right)^{1/(A-1)}\right\}, \cr
C_{\eta} = \left\{ x \in \R \colon x < (1 - K_A) \left(\frac{2\pi |\eta|}{C A}\right)^{1/(A-1)}\right\},
\end{align*}
and rewrite our integral as 
\begin{align*}
\int_{\R} e^{-2 \pi \eta x} e^{-C|x|^A} \, \mmd x & = \int_{A_{\eta}} e^{-2 \pi \eta x} e^{-C|x|^A} \, \mmd x + \int_{B_{\eta}} e^{-2 \pi \eta x} e^{-C|x|^A} \, \mmd x + \int_{C_{\eta}} e^{-2 \pi \eta x} e^{-C|x|^A} \, \mmd x \cr
						  & =: I_1 + I_2 + I_3.
\end{align*}
On the interval over which we integrate in $I_1,$  $-2 \pi \eta x - C|x|^A$ is at most (an absolute constant depending on $A$ times) $|\eta|^{\frac{A}{A-1}}.$ This holds because the center of the interval $A_{\eta}$ is the critical point of 
$-2\pi \eta x - C |x|^A$ where this function attains its maximum. As we know that $|A_{\eta}| \lesssim_A |\eta|^{\frac{1}{A-1}},$ 
it follows that 
\begin{equation}\label{eq:first_bound}
|I_1| \lesssim |\eta|^{\frac{1}{A-1}} e^{C_A |\eta|^{\frac{A}{A-1}}}.
\end{equation}
On either the interval defining $I_2$ or on the one defining $I_3,$ we see that, for $K_A, \tilde{C}_A>0$ large enough depending on $A$, it holds that 
\[
-2\pi \eta x - C|x|^A \le -\tilde{C}_A |x|^A. 
\]
Therefore, 
\begin{equation}\label{eq:second_bound}
|I_2| + |I_3| \lesssim \int_{\eta^{\frac{1}{A-1}}}^{+\infty} e^{-C'_A|x|^A} \, \mmd x \lesssim e^{-C''_A|\eta|^{\frac{A}{A-1}}}. 
\end{equation}
As \eqref{eq:first_bound} dominates \eqref{eq:second_bound}, we obtain that 
\[
|\widehat{f}(z)| \lesssim_A |\eta|^{\frac{1}{A-1}} e^{C_A |\eta|^{\frac{A}{A-1}}}. 
\]
As polynomials factors in $|\eta|$ decay slower than any exponential $e^{|\eta|^{\epsilon}},$ we finish the proof of the result, as $|\eta| \le |z|.$ 
\end{proof}

As an immediate corollary, we obtain the following statement, which will be particularly useful in Section \ref{sec:proof_i}.

\begin{corollary}\label{th:analytic_order1} Let $f \in \mathbb{S}(\R)$ be such that, for each $A >1,$ there is a constant $C_A >0$ such that $|f(x)| \lesssim_A e^{-C_A|x|^A}, \, \forall x \in \R.$ Then its Fourier transform
can be extended to the whole complex plane as an analytic function with order at most 1. 
\end{corollary}

\section{Proof of (A)}\label{sec:proof_i}

\subsection{Obtaining decay for $f$} The first idea is to exploit the considerations in Section \ref{sec:zero_decay} to obtain decay for $f.$ We must, however, obtain decay on the Fourier transform to somehow improve the 
decay on $f$ we obtain at each step. The following Lemma is the key ingredient to this iteration scheme. 

\begin{lemma}\label{th:lemma_decay_log} Let $f \in \mathcal{S}(\R),$ and assume that $f(\pm \log(n+1))=0$ and $\widehat{f}(\pm n^{\alpha})=0$ for every $n \in \N,$ where $\beta \in (0,1).$ Then, for $|x| > \log(k+1)$ and $|\xi| > (2j+1)^{\alpha},$ one has 
\begin{align}\label{eq:decay_logs}
\begin{split}
|f(x)| &\le k(2\pi)^k ((k+1)!)^3 I_k(\widehat{f}) e^{-k|x|} = \tau_k e^{-k|x|},  \cr 
|\widehat{f}(\xi)|  &\leq (j+1)!(2^{2-\alpha}\pi)^{j+1}\alpha^j I_{j}(f)|\xi|^{j\left(\frac{\alpha -1}{\alpha}\right)}=\widehat{C}_j|\xi|^{j\left(\frac{\alpha -1}{\alpha}\right)}.
\end{split}
\end{align}
\end{lemma}

\begin{proof} We first prove the assertion about $\widehat{f},$ as it will be also of interest to Lemma \ref{th:lemma_decay1} in the next section. Let $\xi \geq 0$. First we consider $n$ such that $\xi \in [n^{\alpha},(n+1)^{\alpha}]$. 
This implies $n^{\alpha -1}\leq 2^{1-\alpha}\xi^{\frac{\alpha -1}{\alpha}}$. 
By inequality \eqref{eq:Fourier_inverstion_step}, we have
\begin{align}\label{eq:decay_function} \begin{split}
|\widehat{f}(\xi)|  & \le |(n+1)^{\alpha} - n^{\alpha}| I_1(f) \cr 
 & \le 2\pi \alpha n^{\alpha -1}I_1(f) \cr
 & \le 2^{2-\alpha}\pi \alpha x^{\frac{\alpha -1}{\alpha}}I_1(f).    
\end{split}
\end{align}
Now, by observation (I.i), as long as $\xi>(2j+1)^{\alpha}$, we can conclude there is $n \ge j$ such that $\xi \in [a_{n+1}^{(j)},a_{n}^{(j)}]\subset [n^{\alpha},(n+j+1)^{\alpha}]$. This means $n^{\alpha -1}\leq 2^{1-\alpha}\xi^{\frac{\alpha -1}{\alpha}}$, and therefore
\begin{align} \label{eq:decay_function_derivative}\begin{split}
|[\widehat{f}]^{(j)}(\xi)| & \le (2\pi)^j |a_{n+1}^{(j)}-a_{n}^{(j)}| I_{j+1}(f) \cr 
 & \le (2\pi)^{j+1}|(n+j+1)^{\alpha} - n^{\alpha}| I_{j+1}(f) \cr 
 & \le \alpha(j+2)(2\pi)^{j+1} n^{\alpha -1}I_{j+1}(f) \cr
 & \le 2^{1-\alpha} \alpha(j+2)(2\pi)^{j+1} \xi^{\frac{\alpha -1}{\alpha}}I_{j+1}(f).    
\end{split}
\end{align}
By the fundamental theorem of calculus and inequality \eqref{eq:decay_function_derivative} for $j=1$, we have 
\begin{align} \label{eq:FTC_step1}
\begin{split}
    |\widehat{f}(\xi)| &= |\widehat{f}((n+1)^\alpha)-\widehat{f}(\xi)| \\
 &=  \left|\int^{(n+1)^\alpha}_{\xi} [\widehat{f}]'(y) \, \mmd y \right| \cr 
 & \leq 3 \cdot 2^{1-\alpha} \cdot (2\pi)^2\alpha I_2(f)|(n+1)^{\alpha} - n^{\alpha}|\xi^{\frac{\alpha -1}{\alpha}} \\
 &\leq 3 \cdot 2 \cdot (2^{2-\alpha}\pi)^2\alpha^2 I_2(f)\xi^{2\left(\frac{\alpha -1}{\alpha}\right)}.
\end{split}
\end{align}
Inequality \eqref{eq:FTC_step1} exemplifies how one can use the concentration properties of the sequence $n^{\alpha}$ in order to obtain decay for $f$ and $\widehat{f}$. We can iterate these inequalities for higher order derivatives and obtain better decay. For instance, if we apply the same reasoning as in \eqref{eq:FTC_step1} for the first derivative, we obtain
\begin{align} \label{eq:FTC_step2}
\begin{split}
    |[\widehat{f}]'(\xi)| &= |[\widehat{f}]'(a_{n+1}^{(1)})-[\widehat{f}]'(\xi)| \\
 &=  \left|\int^{a_{n+1}^{(1)}}_{\xi} [\widehat{f}]''(y) \, \mmd y \right| \cr
 &\leq 4\cdot 3 \cdot (2^{2-\alpha}\pi)^3 \alpha^2 I_3(f)\xi^{2\left(\frac{\alpha -1}{\alpha}\right)}.
\end{split}
\end{align}
If we combine this new extra decay for $[\widehat{f}]'$ with the fundamental theorem of calculus, as in \eqref{eq:FTC_step1}, we obtain for $\xi>2$ that
\begin{align} \label{eq:FTC_step3}
\begin{split}
    |\widehat{f}(\xi)| &\leq 4 \cdot 3 \cdot 2 \cdot (2^{2-\alpha}\pi)^4 \alpha^3 I_3(f)\xi^{3\left(\frac{\alpha -1}{\alpha}\right)}.
\end{split}
\end{align}
By induction, one can iterate this process and obtain decay of the order of $\xi^{j\left(\frac{\alpha -1}{\alpha}\right)}$ for $\xi>(2j+1)^{\alpha}$, More precisely,
\begin{align} \label{eq:FTC_stepk}
\begin{split}
    |\widehat{f}(\xi)|  &\leq (j+1)!(2^{2-\alpha}\pi)^{j+1}\alpha^j I_{j}(f)\xi^{j\left(\frac{\alpha -1}{\alpha}\right)}.
\end{split}
\end{align}
Applying the same analysis for negative $\xi$ yields the desired result for $\widehat{f}.$ In order to obtain the asserted bound for $f,$ we run the same scheme of proof, paying attention to the fact that, if $\{b_m^{(k)}\}_{m \in \Z}$ denotes the sequence of zeros of $f,$
in the sense of Section \ref{sec:decay}, then $[b_m^{(k)},b_{m+1}^{(k)}] \subset [\log(m+1),\log(m+k+2)]$, and  
 $$|b_m^{(k)} - b_{m+1}^{(k)}| \le \log (1+\frac{k+1}{m+1}) \le \frac{k+1}{m+1} \le (k+1)^2 \cdot e^{-\log(m+k+1)}.$$
If $x \ge0 $ belongs to the interval $[b_m^{(k)},b_{m+1}^{(k)}],$ then the expression above is bounded by $(k+1)^2 e^{-x}.$ We leave out the details to the iteration procedure, for they 
essentially only replicate equations \eqref{eq:FTC_step1}--\eqref{eq:FTC_stepk}.
\end{proof}
We now describe, in a concise way, the iteration scheme to be undertaken. Since $f\in\mathcal{S}(\R)$, there is a constant $D>0$ such that
\begin{align*}
    |\widehat{f}(\xi)|\leq D.
\end{align*}
Hence
\begin{align*}
    \begin{split}
    I_k(\widehat{f})&\leq D\int_{|\xi|\leq (1+j)^{\alpha}}|\xi|^{k}\mmd \xi+\widehat{C_j}\int_{|\xi|\leq (1+j)^{\alpha}}|\xi|^{k+j\left(\frac{\alpha-1}{\alpha}\right)}\mmd \xi \cr
                    &\leq 2D\frac{1}{k+1}(1+j)^{\alpha(k+1)}+\widehat{C_j}\frac{1}{k+j\left(\frac{\alpha-1}{\alpha}\right)+1}(1+j)^{k+j\left(\frac{\alpha-1}{\alpha}\right)+1}, \cr
    \end{split}
\end{align*}
as long as we choose $j\geq\frac{(k+2)\alpha}{1-\alpha}$. Choosing $j=j(k)\sim \frac{(k+2)\alpha}{1-\alpha}$ implies
\begin{align}\label{eq:bounding_I_k_step_2}
    \begin{split}
    I_k(\widehat{f})&\leq 2D\frac{1}{k+1}(1+\frac{(k+2)\alpha}{1-\alpha})^{\alpha(k+1)}+\widehat{C_j}\frac{1}{2k+2}(1+\frac{(k+2)\alpha}{1-\alpha})^{-1} \\
                   &\leq A_\alpha\left(k^{\alpha(k+1)-1}+\widehat{C}_j\frac{1}{k^2}\right) = A_{\alpha} \left(k^{\alpha(k+1)-1}+(j+1)!(2^{2-\alpha}\pi)^{j+1} \alpha^j I_{j}(f)\frac{1}{k^2}\right).\
    \end{split}
\end{align}
We also observe that \eqref{eq:decay_logs} for $k=1$ implies 
\begin{equation}\label{eq:k=1}
I_j(f) \le \mathcal{C}(f) \int_{\R} e^{-|x|}|x|^j \, \mmd x \lesssim_f j!.
\end{equation}
Putting together \eqref{eq:bounding_I_k_step_2}, \eqref{eq:k=1} together with \eqref{eq:decay_logs}, we obtain that 
\begin{align}\label{eq:final_decay_f}
|f(x)| &\le k(2\pi)^k ((k+1)!)^3 I_k(\widehat{f}) e^{-k|x|}   \cr 
       &\le k(2\pi)^k ((k+1)!)^3 A_{\alpha} \left(k^{\alpha(k+1)-1}+(j+1)!(2\pi\alpha)^j I_{j}(f)\frac{1}{k^2}\right) e^{-k|x|} \cr 
       &\le e^{O(k \, \log k) - k|x|},\cr
\end{align}
for $|x| \ge \log(k+1),$ where by $O(k \, \log k)$ we denote an expression that is bounded by $C_{\alpha} k \log (k+1),$ for some constant depending on $\alpha.$ Equation \eqref{eq:final_decay_f} implies, as $k \le e^{|x|} - 1$ can be chosen 
arbitrarily, that for each $A \gg 1,$ there is $c_A > 0$ such that 
\begin{equation}\label{eq:arbitrary_decay}
|f(x)| \lesssim_{f,A} e^{-c_A |x|^A}.
\end{equation}

\subsection{Viewing $\widehat{f}$ as an entire function} The final part of the argument uses complex analysis to derive a contradiction. In fact, by Corollary \ref{th:analytic_order1}, $\widehat{f}$ is an entire function of order at most 1. The converse to Hadamard's factorisation theorem then predicts that the sum of inverses of zeros of $\widehat{f}$ raised to $1+\varepsilon$ should converge, no matter
which value of $\varepsilon>0$ we choose. But we know that $\{\pm n^{\alpha}\}_{n \ge 0}$ is contained in the set of zeros of $\widehat{f},$ therefore 
\[
\sum_{n \ge 0} \frac{1}{n^{(1+\varepsilon)\alpha}} < +\infty.
\]
This is a clear contradiction, as long as $\alpha <1.$ The contradiction came from assuming that $\widehat{f} \not\equiv 0,$ and thus we have proved the first part of Theorem \ref{th:mainthm}. 

\section{Proof of (B)}\label{sec:proof_ii} 

\subsection{Obtaining simultaneous decay}\label{sec:decay}
The first key step of the proof is obtaining enough decay on $\widehat{f}$ in order extend $f$ as an analytic function. One of the key estimate for that will be an iteration scheme of inequality \eqref{eq:Fourier_inverstion_step}, which is the content of the next Lemmas 
\smallskip
\begin{lemma}\label{th:lemma_decay1}
Let $f\in \mathcal{S}(\R)$ and assume that $f(\pm(n)^{\alpha})=0$ and $\widehat{f}(\pm n^{\beta})=0$ for every $n\in \N$, where $0<\alpha,\beta<1$. Then, for $|x|>(k+1)^{\alpha}$ and $|\xi|>(j+1)^{\beta}$, one has
\begin{align*}
        |f(x)|  &\leq (k+1)!(2^{2-\alpha}\pi)^{k+1} \alpha^k I_k(\widehat{f})|x|^{k\left(\frac{\alpha -1}{\alpha}\right)} =C_k|x|^{k\left(\frac{\alpha -1}{\alpha}\right)} \\
            |\widehat{f}(\xi)|  &\leq (j+1)!(2^{2-\beta}\pi)^{j+1} \beta^j I_{j}(f)|\xi|^{j\left(\frac{\beta -1}{\beta}\right)}=\widehat{C}_j|\xi|^{j\left(\frac{\beta -1}{\beta}\right)}.
\end{align*}
\end{lemma}

The proof of this Lemma is identical to that of Lemma \ref{th:lemma_decay_log}, and we therefore skip it. Lemma \ref{th:lemma_decay1} means that one can get very good decay for $f(x)$ for large values of $x$ by sacrificing the potentially big constant 
\begin{align*}
    C_k=(k+1)!(2^{2-\alpha}\pi)^{k+1}\alpha^k I_k(\widehat{f}) = B_kI_k(\widehat{f}).
\end{align*}
The number $B_k$ is easy to estimate by using Stirling's formula. Indeed
\begin{align}\label{eq:B_k_stirling}\begin{split}
B_k&\leq Ce^{-(k+1)+(k+3/2)\log(k+1)+k\log(2\pi\alpha) }   \\    
   &\leq c_\alpha e^{k\log k+(\log(2\pi\alpha)+1)k+\frac{3}{2}\log{k}}
\end{split}
\end{align}
Meanwhile, the number $I_k(\widehat{f})$, although finite, might grow at an undesirable rate. Our next step is to control the integral $I_k(f)$. 

\begin{lemma}\label{th:decay_I_k}
Let $f\in \mathcal{S}(\R)$ and assume that $f(\pm n^{\alpha})=0$ and $\widehat{f}(\pm n^{\beta})=0$ for every $n\in \N$, where $0<\alpha+\beta<1$. Then there
exists $\tau = \tau(\alpha,\beta) > 0$ such that 
$$I_k(f) \lesssim_{f,\alpha,\beta} e^{-\tau k \log k + O(k)}.$$ 
\end{lemma}
\begin{proof}
From previous considerations, we know that it holds that
\begin{align}\label{eq:bounding_I_k_step_2'}
    \begin{split}
    I_k(\widehat{f})&\leq 2D\frac{1}{k+1}(1+\frac{(k+2)\beta}{1-\beta})^{\beta(k+1)}+\widehat{C_j}\frac{1}{2k+2}(1+\frac{(k+2)\beta}{1-\beta})^{-1} \\
                   &\leq A_\beta\left(k^{\beta(k+1)-1}+\widehat{C}_j\frac{1}{k^2}\right),
    \end{split}
\end{align}
where $|\widehat{f}| \le D$ pointwise. We can now apply the same inequality to $I_k(f)$, and obtain
\begin{align}\label{eq:bounding_I_k_step_3}
    \begin{split}
    I_k(f) &\leq A_\alpha\left(k^{\alpha(k+1)-1}+C_{\widehat{j}}\frac{1}{k^2}\right),
    \end{split}
\end{align}
where $\widehat{j}=\widehat{j}(k)\sim \frac{(k+2)\alpha}{1-\alpha}$. 
Keeping in mind that 
\begin{align*}
    C_k&=(k+1)!(2^{2-\alpha}\pi)^{k+1}\alpha^k  I_k(\widehat{f}) = B_kI_k(\widehat{f})\\
    \widehat{C}_j&=(j+1)!(2^{2-\beta}\pi)^{j+1}\beta^j I_j(f) = \widehat{B}_jI_j(f),
\end{align*}
one can iterate inequalities \eqref{eq:bounding_I_k_step_2} and \eqref{eq:bounding_I_k_step_3}. This means 
\begin{align*}
    \begin{split}
    I_k(f) &\leq A_\alpha\left(k^{\alpha(k+1)-1}+{B}_{\widehat{j}(k)}\frac{1}{k^2}I_{\widehat{j}(k)}(\widehat{f})\right) \\
    &\leq A_\alpha\left(k^{\alpha(k+1)-1}+{B}_{\widehat{j}(k)}\frac{1}{k^2}A_\beta\left({\widehat{j}(k)}^{\beta(\widehat{j}(k)+1)-1}+\widehat{C}_{j(\widehat{j}(k))}\frac{1}{(\widehat{j}(k))^2}\right)\right) \\
    &= A_\alpha\left(k^{\alpha(k+1)-1}+{B}_{\widehat{j}(k)}\frac{1}{k^2}A_\beta\left({\widehat{j}(k)}^{\beta(\widehat{j}(k)+1)-1}+\frac{\widehat{B}_{j(\widehat{j}(k))}}{(\widehat{j}(k))^2}I_{j(\widehat{j}(k))}(f)\right)\right).
    \end{split}
\end{align*}
This chain of inequalities amounts to the following inequality
\begin{align}\label{eq:bounding_I_k_step_4}
    \begin{split}
    I_k(f) &\leq G(k)+H(k)I_{j(\widehat{j}(k))}(f),
    \end{split}
\end{align}
where
\begin{align}\label{eq:G_H_order}\begin{split}
    G(k)&=A_{\alpha,\beta}(k^{\alpha(k+1)-1}+{B}_{\widehat{j}(k)}\frac{1}{k^2}{\widehat{j}(k)}^{\beta(\widehat{j}(k)+1)-1}) \\
    H(k)&=A_{\alpha,\beta}\frac{{B}_{\widehat{j}(k)}\widehat{B}_{j(\widehat{j}(k))}}{k^2\widehat{j}(k)^2}.
    \end{split}
\end{align}
An observation in order is that
\begin{align*}
    \rho(k)=j(\widehat{j}(k))\sim \left(\tfrac{\alpha}{1-\alpha}\right)\left(\tfrac{\beta}{1-\beta}\right)k,
\end{align*}
and 
\begin{align*}
    \gamma=\left(\tfrac{\alpha}{1-\alpha}\right)\left(\tfrac{\beta}{1-\beta}\right)<1 \, \Leftrightarrow \, \alpha + \beta <1.
\end{align*}
Since we assumed that $\alpha+\beta < 1$, it implies $\gamma <1$ and inequality \eqref{eq:bounding_I_k_step_4} roughly translates
\begin{align}\label{eq:bounding_I_k_step_5}
    \begin{split}
    I_k(f) &\leq G(k)+H(k)I_{\gamma k}(f),
    \end{split}
\end{align}
and by iterating one gets
\begin{align}\label{eq:bounding_I_k_step_6}
    \begin{split}
    I_k(f)   &\leq\sum_{l=0}^{m-1}\left[G(\gamma^lk)\prod_{s=0}^{l-1}H(\gamma^{s}k)\right] +H(\gamma^{m-1}k)\cdots H(\gamma k)H(k)I_{\gamma^m k}(f).
    \end{split}
\end{align}
In order for our bounds to behave nicely, we assume at this point that $A_{\alpha,\beta}=1$ in \eqref{eq:G_H_order}, which is possible simply by dividing $f$ by $A_{\alpha,\beta}$ at the cost of an extra constant depending only on $\alpha$ and $\beta$ on the desired bounds. 
We estimate $G$ using \eqref{eq:B_k_stirling}
\begin{align*}
    G(k)&\lesssim_{\alpha} e^{\alpha(k+1)\log{k}}+  e^{(1+\beta)\frac{\alpha}{1-\alpha}k\log k+O(k)}\cr
        &\leq e^{\lambda k\log k+E(k)},
\end{align*}
where 
\begin{align*}
    \lambda=(1+\beta)\frac{\alpha}{1-\alpha},
\end{align*}
and $E(k)=O(k)$. Now we estimate $H$ in the same fashion
\begin{align*}
    H(k)&=\frac{{B}_{\widehat{j}(k)}\widehat{B}_{j(\widehat{j}(k))}}{k^2\widehat{j}(k)^2} \\
        &\lesssim_{\alpha} e^{(\frac{\alpha}{1-\alpha})k\log{k}+O(k)}e^{\gamma k\log{k}+O(k)} \\
        &\leq e^{\delta k\log{k}+E(k)},
\end{align*}
where 
\begin{align*}
    \delta=\frac{\alpha}{1-\alpha}+\gamma = \frac{\alpha}{(1-\alpha)(1-\beta)},
\end{align*}
and $F(k)=O(k)$. This means
\begin{align*}
    \prod_{s=0}^{l-1}H(\gamma^{s-1}k)&\leq e^{\sum_{s=1}^{l-1}[\delta \gamma^sk\log{\gamma^sk}+E(\gamma^sk)]} \\
    & \leq e^{\delta\frac{1-\gamma^{l}}{1-\gamma}k\log{k}+E_0(k)}
\end{align*}
and therefore
\begin{align}\label{eq:bounding_I_k_step_7}
    \begin{split}
    I_k(f) &\leq\left[\sum_{l=0}^{m-1}G(\gamma^lk)\prod_{s=0}^{l-1}H(\gamma^{s}k)\right] +H(\gamma^{m-1}k)\cdots H(\gamma k)H(k)I_{\gamma^m k}(f) \\
    &\leq \left[\sum_{l=0}^{m-1}   e^{\lambda {\gamma^{l}k}\log k+F(\gamma^{l}k)} e^{\delta\frac{1-\gamma^{l}}{1-\gamma}k\log{k}+E_0(k)}\right] +e^{\delta\frac{1-\gamma^{m}}{1-\gamma}k\log{k}+E_0(k)}I_{\gamma^m k}(f) \\
    &\leq m   e^{(\lambda+\delta)\frac{1}{1-\gamma}k\log k+F_0(k)}    +e^{\delta\frac{1}{1-\gamma}k\log{k}+E_0(k)}I_{\gamma^m k}(f).
    \end{split}
\end{align}
Now, if we choose $m\sim -\log_\gamma{k}$, and for simplicity assume $I_1(f)=1$, we have \footnote{Up to this point, we have neglected the error terms ($E$, $F$ etc), but their sums with argument $\gamma^l k$ are clearly still going to be $O(k)$.}
\begin{align*}
    I_k(f)\leq e^{\frac{\lambda+\delta}{1-\gamma}k\log{k} +O(k)}.
\end{align*}
The proof of the Lemma is then complete by taking $\tau = \frac{\lambda + \delta}{1-\gamma}.$ This choice is going to be important for us later on.
\end{proof} 

One direct consequence of Lemma \ref{th:decay_I_k} is that we obtain an explicit decay for $\widehat{f}$ of the form 
\begin{align}\label{eq:decay_Fourier2}
\begin{split}
    |\widehat{f}(\xi)|&\leq e^{(1+\frac{\lambda+\delta}{1-\gamma})k\log{k}+O(k)}|\xi|^{k\left(\frac{\beta -1}{\beta}\right)} \\
    &=e^{(1+\frac{\lambda+\delta}{1-\gamma})k\log{k}+\left(\frac{\beta -1}{\beta}\right)k\log{|\xi|}+O(k)},
\end{split}
\end{align}
whenever $(1+2k)^{\beta}\leq |\xi|$. Now, if one chooses $k\sim|\xi|^\frac{1}{\epsilon}$, the exponent in \eqref{eq:decay_Fourier2} becomes
\begin{align*}
    \left[\frac{1}{\epsilon}\left(1+\frac{\lambda+\delta}{1-\gamma}\right)\log{|\xi|}+\left(\frac{\beta -1}{\beta}\right)\log{|\xi|}\right]|\xi|^{\frac{1}{\epsilon}}+O(|\xi|^{\frac{1}{\epsilon}}).
\end{align*}
As long as 
\begin{align*}
    \frac{1}{\epsilon}\left(1+\frac{\lambda+\delta}{1-\gamma}\right)<\frac{1 -\beta}{\beta},
\end{align*}
or equivalently
\begin{align}\label{eq:size_condition}\begin{split}
    \epsilon&>\left(1+\frac{\lambda+\delta}{1-\gamma}\right)\frac{\beta}{1-\beta}\\
    &=\frac{1-\alpha - \beta + (2-\beta^2)\alpha}{1-\alpha - \beta}\frac{\beta}{1-\beta}\\
    &=\frac{1+\alpha - \beta(1+\alpha \beta)}{1-\alpha-\beta} \cdot \frac{\beta}{1-\beta},
\end{split}
\end{align}
we can conclude that, for some $0<\theta<1$,
\begin{align}\label{eq:first_exponential_decay}
    |\widehat{f}(\xi)|&\lesssim_{f} e^{-(1-\theta)|\xi|^{\frac{1}{\epsilon}}}, 
\end{align}
where $\theta>0$ is small, and \eqref{eq:size_condition} is obviously true for some admissible large $\epsilon$, i.e, some number such that $(1+2k)^\beta<|\xi|\sim k^{\epsilon}$. We next connect this exponential decay we have
achieved to the magnitude of $I_k(f).$ 

\begin{lemma}\label{th:I_k_improvement} Let $f \in \mathcal{S}(\R)$ such that 
\begin{equation}\label{eq:exponential_decay_f} 
|f(x)| \le C_f e^{-(1-\theta)|x|^{\frac{1}{\delta}}}.
\end{equation}
Then it holds that $I_k(f) \lesssim_{f,\delta,\theta} \Gamma(\delta (k+1)).$ 
\end{lemma}
\begin{proof} By \eqref{eq:exponential_decay_f}, it follows that
\begin{align*}
    I_k(f)\lesssim \int_\R e^{-(1-\theta)|x|^{\frac{1}{\delta}}}|x|^{k}\mmd x. 
\end{align*}
By the change variables $x\leadsto \frac{t^\varepsilon}{(1-\theta)^\varepsilon}$, we have
\begin{align*}
    \begin{split}
        \int_\R e^{-|x|^{\frac{1}{\varepsilon}}}|x|^{k}\mmd x =\frac{2\varepsilon}{(1-\theta)^{k(\varepsilon+1)}}\int_0^\infty e^{-t}t^{\varepsilon(k+1)-1}\mmd t =\frac{2\varepsilon}{(1-\theta)^{k(\varepsilon+1)}}\Gamma(\varepsilon(k+1)),
    \end{split}
\end{align*}
which directly implies the assertion of the Lemma.
\end{proof} 

\subsection{Optimizing the exponent}\label{sec:optimal_exp} It is important to point out that up to this point the only imposed condition on the pair $(\alpha,\beta)$ is that $\alpha+\beta<1$. This means that whenever $f$ is a Schwartz function such that $f(\pm n^{\alpha})=0$ and $\widehat{f}(\pm n^{\beta})=0$, then inequality  \eqref{eq:first_exponential_decay} holds for some small $\theta$ and $\varepsilon$ satisfying \eqref{eq:size_condition}. 
We now describe an iteration procedure to improve the decay obtained in the previous subsection, at the cost of extra constraints on the pair $(\alpha,\beta)$.

\smallskip

Let $\epsilon(\widehat{f})$ denote the infimum of all $\epsilon >0$ obtained previously, such that \eqref{eq:first_exponential_decay} holds. That is, 
we let 
\[
\epsilon(\widehat{f}) = \frac{\beta(1+\frac{\lambda+\delta}{1-\gamma})}{1-\beta}.
\]
Define $\epsilon(f)$ in the same fashion, exchanging the roles of $\alpha$ and $\beta.$ The process that follows is a way to progressively decrease the magnitude of either $\epsilon(f)$ or $\epsilon(\widehat{f}).$

\smallskip

It follows from Lemma \ref{th:I_k_improvement} that
\begin{align*}
       |\widehat{f}(x)|  \leq e^{(1+\epsilon(f))k\log{k}+\left(\frac{\beta -1}{\beta}\right)k\log{|\xi|}+O(k)}.
\end{align*}
Define then the sequences $(a_n,b_n)_{n \in \Z}$ of exponents associated to $f,\widehat{f}$ to be 
\begin{align}\label{eq:recursion}
b_0 = \epsilon(\widehat{f}), & \, a_0 = \epsilon(f), \cr 
b_{n} = (1 + a_{n})\frac{\beta}{1-\beta}, \, & a_{n+1} = (1 + b_n)\frac{\alpha}{1-\alpha}. \cr
\end{align}
Repeating the argument undertaken in Section \ref{sec:decay}, it holds that, for any $\varepsilon > 0,$
\begin{align*}
|f(x)| & \lesssim_{f,\varepsilon} e^{-C_n|x|^{\frac{1}{a_n+\varepsilon}}}, \cr 
|\widehat{f}(\xi)| & \lesssim_{f,\varepsilon} e^{-\tilde{C}_n|\xi|^{\frac{1}{b_n+\varepsilon}}}, \cr 
\end{align*}
as long as the conditions $b_n > \beta$ and $a_n > \alpha$ are met for all $n \ge 0.$ We let, respectively to the definitions above, 
\begin{align*}
\theta_1(\alpha,\beta) & = \frac{\alpha}{(1-\alpha)(1-\beta)}, \cr
\theta_2(\alpha,\beta) & = \frac{\beta}{(1-\alpha)(1-\beta)}. 
\end{align*}
A computation shows that we actually have 
\begin{align}\label{eq:recursion_2}
a_{n+1} = \theta_1 + \gamma a_n, \cr
b_{n+1} = \theta_2 + \gamma b_n. 
\end{align}
As $\gamma < 1,$ we see that both $(a_n)_{n \ge 0}$ and $(b_n)_{n \ge 0}$ are convergent sequences, with limit 
\begin{align*}
L_1(\alpha,\beta) = \lim_{n \to \infty} a_n = \frac{\alpha}{1-\alpha-\beta}, \cr 
L_2(\alpha,\beta) = \lim_{n \to \infty} b_n = \frac{\beta}{1-\alpha - \beta}. \cr 
\end{align*}
This implies that, for all $\varepsilon > 0,$ 
\begin{align}\label{eq:best_exponential_decay}
|f(x)| & \lesssim_{f,\varepsilon} e^{-C|x|^{\frac{1}{L_1(\alpha,\beta)+\varepsilon}}}, \cr 
|\widehat{f}(\xi)| & \lesssim_{f,\varepsilon} e^{-\tilde{C}|\xi|^{\frac{1}{L_2(\alpha,\beta)+\varepsilon}}}. \cr 
\end{align}
Notice that, if $\epsilon(f) > L_1(\alpha,\beta)$ and $\epsilon(\widehat{f}) > L_2(\alpha,\beta),$ then both sequences $a_n,b_n$ are \emph{decreasing}, and \eqref{eq:best_exponential_decay} is
the best exponential decay we could expect for $f, \widehat{f}.$ Notice that the condition \eqref{eq:size_condition} gives us that $\epsilon(\widehat{f}) > L_2(\alpha,\beta)$ as desired, which 
proves that the iteration scheme presented achieves, in fact, a better exponential decay for $f, \widehat{f}$ than the original one. 

\begin{rmk} If we let $S^{\nu}_{\mu}(\R)$ denote the \emph{Gelfand-Shilov space} of Schwartz functions $\varphi$ such that 
\[
\sup_{x \in \R} |\varphi(x) e^{h|x|^{1/\nu}}|,\,\,\, \sup_{\xi \in \R} |\widehat{\varphi}(\xi) e^{k|\xi|^{1/\nu}}| < + \infty
\]
for some $k,h>0,$ then we have actually proved that $f \in \tilde{S}^{\nu}_{\mu}(\R) := \cup_{\nu_0 > \nu,\mu_0>\mu} S^{\nu_0}_{\mu_0}(\R),$ where $\nu = L_1(\alpha,\beta)$ and $\mu = L_2(\alpha,\beta).$ These function spaces are originally defined through specific decay 
properties of the Schwartz seminorms $\varphi \mapsto \|x^{\alpha} \partial^{\beta}\varphi\|_{\infty},$ and the equivalence to the higher-order decay statement above is proved through the seminorm decay. This 
procedure is in many ways analogous to the one undertaken here to obtain that $f \in S^{\nu}_{\mu}(\R),$ and the relationship between our proof and these function spaces was recently brought to our attention. For 
more information on Gelfand-Shilov spaces, see, for instance, \cite{CCK,GS1968} and the references therein. 
\end{rmk}

\subsection{Analytic continuation}\label{sec:analytic_continuation2} We wish to derive a contradiction from the fact that $f \not\equiv 0.$ In order to do it, we prove that either $f$ or $\widehat{f}$ can be analytically extended 
with control on its order depending only on $\min\{L_1(\alpha,\beta),L_2(\alpha,\beta)\}.$ Without loss of generality, let $\alpha \le \beta.$ Therefore, $L_1(\alpha,\beta) < L_2(\alpha,\beta)$  and, 
in case $\beta \le 1-2\alpha,$ then $L_1(\alpha,\beta) < 1,$ and this contains the region $A$ described in the introduction. We then appeal to Lemma \ref{th:analytic_continuation}, which enables us to 
conclude that $\widehat{f}$ is extendable as an analytic function of order at most 
$$\frac{1}{1-L_1(\alpha,\beta)}.$$ 
By the converse to Hadamard's factorisation theorem, we must have 
\begin{align*}
\sum_{n \ge 0} n^{-\frac{\beta+\varepsilon}{1-L_1(\alpha,\beta)}} < + \infty,
\end{align*}
for each $\varepsilon > 0.$ Thus, we reach an immediate contradiction if 
\[
\beta < 1- L_1(\alpha,\beta).
\]
As we supposed initially that $\alpha \le \beta,$ elementary calculations lead to the following observation: if $(\alpha,\beta) \in A,$ then 
each Schwartz function $f$ such that $f(\pm n^{\alpha}) = \widehat{f}(\pm n^{\beta}) = 0, \, \forall n \in \mathbb{N},$ then $f \equiv 0.$ This finishes the proof of 
Theorem \ref{th:mainthm}.



\section{Remarks and complements}\label{sec:comments}


\subsection{Spacing between zeros and bounds for $f$} 

In Sections \ref{sec:prelim}, \ref{sec:proof_i} and \ref{sec:proof_ii}, we have seen how to obtain decay for a Schwartz function 
given we have information on the location of the zeros of its derivatives. A main feature, in particular, of the proof in Section \ref{sec:proof_ii} was that the sequence of 
zeros of the derivative $f^{(k)}$ satisfies $a^{(k)}_n \in [n^{\alpha},(n+k+1)^{\alpha}],$ which enables us to bound 
\begin{equation}\label{eq:gap1}
|a^{(k)}_{n+1} - a^{(k)}_n| \le C_{\alpha} (k+1) |a^{(k)}_{n+1}|^{-\frac{1-\alpha}{\alpha}},
\end{equation}
if $n > k+1.$ A careful look into the proofs undertaken above relates the exponent of $k$ on the left hand side above to the iteration scheme for optimizing the exponent performed in Section \ref{sec:optimal_exp}. 
Indeed, if we were able to improve the factor on the right hand side of \eqref{eq:gap1} from $(k+1)$ to $(k+1)^{\omega}, \omega <1,$ then the sequences $a_n,b_n$ above would take the form 
\begin{align}\label{eq:recursion2}
b_0 = \epsilon(\widehat{f}), & \, a_0 = \epsilon(f), \cr 
b_{n} = (\omega + a_{n})\frac{\beta}{1-\beta}, \, & a_{n+1} = (\omega + b_n)\frac{\alpha}{1-\alpha}. \cr
\end{align}
A simple computation shows that the limit of this new sequences is strictly \emph{smaller} than the one we obtained in Section \ref{sec:optimal_exp}. This yields, as a consequence, an improvement 
on the set $A$ of admissible exponents for Theorem \ref{th:mainthm}, described in the introduction. For instance, if \eqref{eq:recursion2} holds, then 
\begin{align*}
\lim_{n \to \infty} a_n = \frac{\omega\alpha(1+(\omega-1)\beta)}{1-\alpha-\beta}, \cr 
\lim_{n \to \infty} b_n = \frac{\omega\beta(1+(\omega-1)\alpha)}{1-\alpha - \beta}. \cr 
\end{align*}
If $\alpha \le \beta$ and $\omega$ satisfies the equation 
\[
\omega(1+(\omega-1)\beta) = 1- \alpha - \beta,
\]
then the argument in Section \ref{sec:analytic_continuation2} produces a contradiction whenever $\alpha + \beta < 1,$ which would be the biggest regime in which one expects a version 
of our main theorem to hold. This raises the question whether the decay in \eqref{eq:gap1} can be improved. Unfortunately, the answer to this question is negative. Indeed, let $a^{(0)}_n = n^{\alpha}$ as before. 
Consider $\{ n \in \N \colon n^\alpha \in [2^j,2^{j+1})\}=[n_j,n_{j+1}),$ and define the sequence $\{a^{(k)}_n\},$ for $n \in [n_j,n_{j+1}-k)$ and $\frac{1}{j+1}2^{k/\alpha} < k <  2^{j/\alpha},$ satisfying 
\begin{align}\label{eq:sequence_old}
 & a^{(k-1)}_{n_j} <  a^{(k)}_{n_j} < (n_j+1)^{\alpha}, \cr 
a^{(k-1)}_{n+1} > & a^{(k)}_n > \max(a^{(k-1)}_{n+1} - 2^{-10k(1-\alpha)j/\alpha},a^{(k-1)}_n).\cr
\end{align}
This satisfies, in particular, the growth requirements on the sequence from Section \ref{sec:zero_decay}. For $k > 2^{j/\alpha}, n \in [n_j,n_{j+1}),$ we let $a^{(k)}_n$ be chosen arbitrarily satisfying (I.i) in the same section.
The definition implies, in particular, that 
$a^{(k)}_{n_j+1} > a^{(0)}_{n_j+k+1} -  \sum_{\ell \le k} 2^{-10\ell(1-\alpha)j/\alpha} > (n_j + k +1)^{\alpha} - c_{\alpha}2^{-10(1-\alpha)j/\alpha}.$ Therefore, 
$$|a^{(k)}_{n_j+1} - a^{(k)}_{n_j}| \ge (n_j + k + 1)^{\alpha} - (n_j + 1)^{\alpha} - 2^{-10(1-\alpha)j/\alpha} \ge \alpha k \cdot (n_j+k+1)^{\alpha-1} - 2^{-10(1-\alpha)j/\alpha}.$$
As $n_j > 2^{j/\alpha},$ the right hand side is controlled from below by a constand depending on $\alpha$ times $k2^{-\frac{(1-\alpha)j}{\alpha}}.$ As $n_{j+1} \le 2^{1/\alpha} 2^{j/\alpha},$
estimate \eqref{eq:gap1} is sharp for $k<2^{j/\alpha}.$ Replicating the same argument for all $j > 1$ and concatenating the sequences together implies the desired sharpness for all $k \ge 1.$ 

\smallskip
Nevertheless, a question still remaining is whether a decay better than \eqref{eq:gap1} can hold \emph{on average}. We have used this estimate on the gap between zeros of the $k-$th derivative
to obtain decay for $f^{(k)}$ pointwise. It could happen, though, that one obtains better decay averaging over large intervals, rather than doing pointwise evaluation. This intuitive thought 
is partially backed up by the fact that, for $n \in [n_j,n_{j+1}-k),$ the average gap 
$$|a^{(k)}_{n+1} - a^{(k)}_n|$$
is of the same order of $2^{-(1-\alpha)j/\alpha},$ as long as $n-k \sim 2^{j/\alpha}.$ We show here that this phenomenom does not happen in case the sequence
of zeros $\{a^{(k)}_n\}$ has structure similar to the counterexample above. Considering the bound \eqref{eq:Fourier_inverstion_step}, we wish to bound the average of $f^{(k)}$ over the interval $[2^j,2^{j+1}).$ A computation shows that
\begin{equation}\label{eq:average_bound}
\dashint_{2^j}^{2^{j+1}} |f^{(k)}(x)| \, \mmd x \lesssim \frac{1}{2^j} (2\pi)^k I_{k+1}(\widehat{f}) \left(\sum_{l=n_j-k}^{n_{j+1}} |a^{(k)}_{l+1} - a^{(k)}_l|^2\right).
\end{equation}
Notice that each of the $|a^{(k)}_{l+1}-a^{(k)}_l|$ terms is bounded by $C_{\alpha}\cdot (k+1) 2^{-(1-\alpha)j/\alpha},$ for some absolute $C_{\alpha}>0.$ Our problem is equivalent to the following: we have a sequence of $N$ non-negative
real numbers $\{c_j\}_{j=1}^N$ such that $\sum_{j=1}^N c_j = A$ and $0 < c_j \le B.$ What is the maximum of 
\begin{equation}\label{eq:sum_squares}
\sum_{j=1}^N c_j^2 ,
\end{equation}
and when is it attained? By fixing all but 2 variables, it is easy to see that the maximum of \eqref{eq:sum_squares} happens when the $c_j$ are all either $B$ or $0$. As 
\[
\sum_{j=1}^N c_j = A,
\]
it holds that the optimal value happens when there are $\sim A/B$ different $j'$s for which $c_j = B$, and then the maximal value of \eqref{eq:sum_squares} is $\sim B \cdot A.$ Applying this analysis to \eqref{eq:average_bound} yields that 
\begin{equation}\label{eq:average_does_not_work}
\dashint_{2^j}^{2^{j+1}} |f^{(k)}(x)| \, \mmd x \lesssim (2\pi)^k I_{k+1}(\widehat{f}) C_{\alpha} \cdot (k+1) 2^{-(1-\alpha)j/\alpha}, 
\end{equation}
as long as $k \le 2^{j/\alpha},$ which is essentially the same as we obtained before. In order to prove that there is a sequence with the behaviour described above, we define a sequence $\{a^{(k)}_n\}$ of the following form: on the interval $[n_j,n_j+k+1),$ 
we define our sequence exactly as in \eqref{eq:sequence_old}; we then do the same construction as in \eqref{eq:sequence_old} on $[n_j+k+1,n_j+2(k+1))$, but with $n_j+k+1$ in place of $n_j.$ Similarly, we do it for each of the $\sim 2^{j/\alpha}/k$ intervals 
of the form $[n_j + \ell(k+1),n_j+(\ell+1)(k+1)).$ The sequence obtained that way will nearly maximise the square sums, in the sense that there are going to be $\sim 2^{j/\alpha}/k$ terms close to $ \sim k 2^{-(1-\alpha)j/\alpha},$ and the remaining ones 
will be close to zero. A computation shows that the bound \eqref{eq:average_does_not_work} holds in the same way for this sequence. 

\smallskip
These examples indicate that not much more can be improved in our methods in terms of the range of exponents $A$ above without additional information about the location of the sequences of zeros $\{a^{(k)}_n\}_{k\ge 0, n \in \Z}.$  


\subsection{Generalisations of Theorem \ref{th:mainthm}}

\subsubsection{Conditions on the sets of zeros}\label{sec:cond_zeros}
One might wonder if the sequences in Theorem \ref{th:mainthm} being composed of powers and logarithms of integers plays an important role in our proofs, but it does not. The spacing of the zeros comes into the proofs in order to produce the first decay estimates, and for that the important piece of information that plays a role is the bound \eqref{eq:gap1}, which comes from the distance between two consecutive zeros of the derivatives of $f$, and the growth condition of the sequence of zeros of $f$ and $\widehat{f}$. In other words, if $f(\pm a_n)=f(\pm b_n)=0$, then it is sufficient to have two positive numbers  $\eta$ and $\omega$ such that
\begin{align}\label{eq:final_eq}\begin{split}
\eta\cdot\omega&>1, \\
|a_{k+n}-a_n|&\leq Ck|a_{k+n}|^{-\eta}, \\
|b_{k+n}-b_n|&\leq Ck|b_{k+n}|^{-\omega},
\end{split}
\end{align}
in order to apply the same procedure as in Lemma \ref{th:decay_I_k} and obtain the initial degree of exponential decay. Now, in order to optimise the exponent as in subsection \ref{sec:optimal_exp}, we need
\begin{align}\label{eq:final_eq2}\begin{split}
|a_n|&\leq Cn^{\frac{1}{1+\eta}}, \\
|b_n|&\leq Cn^{\frac{1}{1+\omega}},
\end{split}
\end{align}
where $(\alpha,\beta)=(\tfrac{1}{1+\eta},\tfrac{1}{1+\omega})$ belong to the region $A$ in Theorem \ref{th:mainthm}. This means our results are stable under small perturbations of the sequences of zeros. In fact one can even delete a large number of zeros and still get the same results. One should compare, for instance, to the interpolation result  \eqref{eq:interpolation_schwartz} mentioned in the introduction, 
whose proof, to the best of our knowledge, is rigid to the fact that the interpolation nodes are the square roots of the natural numbers, and the construction of the interpolation basis itself shows that one cannot remove any term from the sequence without breaking down the final result.

\subsubsection{Conditions on the functions} Another very natural question that arises from the results is if it is completely necessary to assume the functions involved are in the Schwartz class. Perhaps the result could hold with more relaxed conditions, but our proof rely heavily on finiteness of $I_k(f)$ and $I_k(\widehat{f})$ for every $k\geq 0$, and this implies, although not in a straightforward manner, that $f$ is a Schwartz function. For the sake of completeness, we outline the proof of this fact.

\smallskip
First of all, by Fourier inversion and the Riemann-Lebesgue lemma, finiteness of $I_k(\widehat{f})$ implies that $f$ is of $C^\infty$ class with all derivatives bounded and converging to zero at infinity. Now, we only need to prove polynomial decay of all the derivatives of $f$, and in order for that to be true we start by proving that $f$ has polynomial decay. For a fixed $N>0$, we define the set
$$E_{j,N}=E_j=\{x\in [2^j,2^{j+1}): |x|^Nf(x)>1\}.$$
It follows from Chebychev`s inequality that 
\begin{align*}
|E_j|\leq \int_{2^j}^{2^{j+1}}|f(x)||x|^N\dx\leq 2^{-jN}I_{2N}(f),
\end{align*}
This means there is $y\in E_j$ and $x\in [2^j,2^{j+1})\backslash E_j$ such that $|x-y|\leq 2^{-jN}I_{2N}(f)$. By the aforementioned fact that $f'$ is bounded, we have
\begin{align*}
|f(y)|&\leq |f(x)-f(y)|+|f(x)| \\
       &\leq C_f|x-y|+|x|^{-N} \\
       &\lesssim_{N,f} |y|^{-N}.
\end{align*}
Therefore $f$ has polynomial decay of any order. Now, in order to propagate this decay to every derivative, we combine the fact that $f''$ is a bounded function and $|f(x)|\lesssim |x|^{-N}$ with a Taylor series remainder argument in order to obtain $|f'(x)|\lesssim |x|^{-N/2}$. This implies polynomial decay for $f'$. Iterating this argument with higher order derivatives implies that $f$ is of Schwartz class.

\subsubsection{{Radial versions for higher dimensions}}{ A very natural generalisation one could think of is that of asking the same question for higher dimensional functions. Of course the notion of density would have to be redefined for general functions of several variables since one can easily construct functions that vanish along uncountable sets, such as manifolds, but if one restricts its attention to the case of radial functions similar questions will naturally arise. {In fact, if we consider $\mathcal{S}_{rad}(\R^d)$ to be the class of radial Schwartz class on $\R^d$, in \cite{CKMRV2} the authors study interpolation formulas in this radial setting, and dimensional differences come into the fold.} This motivates the question: for which exponents $(\alpha,\beta)$ does the pair $(\{n^\alpha\}_{n\in\Z_+},\{n^\beta\}_{n\in\Z_+})$ forms a Fourier uniqueness pair for $\mathcal{S}_{rad}(\R^d)$? Turns out in our setting the same ideas already introduced here apply to this problem, and we outline the steps here. 
}

\smallskip

{\emph{Step 1}: By replacing $f^{(k)}$ by the $k$-th order radial derivative $\partial_r^kf$, one can run the same game of intermediate zeros as in section \ref{sec:zero_decay} to get high order polynomial decay with loss on the constants involved in terms of $I_{k,d}(f)$ and $I_{k,d}(\widehat{f})$, where
\begin{align*}
    I_{k,d}(g)=\int_{\R^d}|g(x)||x|^k\mmd x.
\end{align*}
One can also obtain analogues of Lemmas \ref{th:decay_I_k} and \ref{th:I_k_improvement}. More precisely, one gets the analogue of inequality \eqref{eq:bounding_I_k_step_4} paying a dimensional constant, which means one can directly replicate Lemma \ref{th:decay_I_k} to obtain
\begin{align}\label{eq:first_exponential_decay_radial}
    |\widehat{f}(|\xi|)|&\lesssim_{f} e^{-(1-\theta)|\xi|^{\frac{1}{\epsilon}}}.
\end{align}
Lemma \ref{th:I_k_improvement} for the $d$-dimensional setting will read as the estimate
\begin{align*}
    I_{k,d}(f)\lesssim_{f,\delta,\theta} \Gamma(\delta(k+d)),
\end{align*}
which can be applied in the same fashion in the rest of the iteration procedures to reach the same order of decay.}

\smallskip

\smallskip
{\emph{Step 2}: Hadamard's theorem on distribution of zeros of entire functions fails to work in the same fashion for several complex variable functions, so one cannot do the simply extend the radial functions involved to $\C^d$. The alternative to this is observe that the Fourier transform of a radial function can be seen as a Hankel transform. We consider the following Hankel transform 
\begin{align*}
 \mathcal{H}_{\nu}(f)(\rho):=\int_0^{\infty}f(r)A_{\nu}(r\rho)\mmd r,
\end{align*}
where $\mathcal{A}_\nu(s)=(2\pi s)^{\nu}J_{\nu}(2\pi s)$, and $J_\nu$ is a Bessel function of first kind. In this setting, if we consider $\widetilde{f}(r)=f(r)r^{d-1}$, which has the same zeros as $f$, then 
\begin{align*}
 \widehat{f}(\xi)=(2\pi)^{\frac{d}{2}}\mathcal{H}_{\frac{d-2}{2}}(\widetilde{f})(|\xi|).
\end{align*}
By observing that the function $\mathcal{A}_{\frac{d-2}{2}}$ can be extended as a real entire function satisfying the estimate
\begin{align*}\label{eq:bessel_growth}
    |\mathcal{A}_{\frac{d-2}{2}}(\xi+i\eta)|\lesssim_d e^{2\pi|\eta|},
\end{align*}
it is clear that an analogue version of Lemma \ref{th:analytic_continuation} holds for the Hankel transform. 
}

\smallskip
{\emph{Step 3}: In order to finish, we now combine the analytic extension property of the Hankel transform and its connections with the Fourier transform mention in Step 2, together with the decay mentioned in Step 1, one can invoke Hadamard's theorem in the same fashion as before and conclude $f$ has to be the zero function, as long as $(\alpha,\beta)\in A$, where $A$ is the set introduced in Theorem \ref{th:mainthm}.
}

\subsection{Open problems} Comparing Theorem \ref{th:mainthm} and \eqref{eq:interpolation_schwartz}, we see that there is a gap in area between the two pictures. The $(\alpha,\beta) = (1/2,1/2)$ point considered by Radchenko and Viazovska possesses a `quasi-uniqueness' property, in the sense that there is essentially one real function who vanishes on the nodes $\pm \sqrt{n}$ and 
belongs to the Schwartz class. We believe that the question of denseness of the sequences $(\pm n^{\alpha},\pm n^{\beta})$ plays an important role in removing this rigidity condition, which is reflected on the following conjecture.

\begin{conjecture}\label{conj:full_range} Let $\alpha, \beta \in (0,1)$ be such that $\alpha + \beta < 1.$  If $f \in \mathcal{S}(\R)$ satisfies that $f(\pm n^{\alpha}) = \widehat{f}(\pm n^{\beta}) = 0$ for all $n \ge 0,$ then it holds that $f \equiv 0.$ 
\end{conjecture}

Of course, Theorem \ref{th:mainthm} is partial progress towards this conjecture, but our techniques do not seem to be immediately susceptible to being generalised in order to conclude the full conjecture. On the other hand, another interesting problem that, as far as we know, is still largely unexplored is that of sequences that grow roughly as a power of an integer, but do not posses as strong tightness properties 
as in Section \ref{sec:cond_zeros} above. 

\begin{question} Let $\alpha, \beta \in (0,1)$ be such that $\alpha + \beta < 1.$ Under which conditions does it hold that, for two sequences $(\pm c_n, \pm d_n)_{n \ge 0}$ such that 
\begin{align*}
\lim_{n \to \infty} \frac{d_n}{n^{\beta}}, \lim_{n \to \infty} \frac{c_n}{n^{\alpha}} < + \infty
\end{align*}
and a function $f \in \mathcal{S}(\R)$ such that $f(\pm c_n) = \widehat{f}(\pm d_n) = 0, \forall n \ge 0,$ then $f \equiv 0$? 
\end{question}
The first natural guess is that a result of that kind should hold in the same range as Conjecture \ref{conj:full_range}, but it would already be interesting if one could prove that the uniqueness property holds under the assumptions in Theorem \ref{th:mainthm}. Finally, our last question concerns what happens on the critical case of Theorem \ref{th:mainthm}.

\begin{question} Let $\alpha, \beta \in (0,1)$ be such that $\alpha + \beta = 1.$ Suppose $f \in \mathcal{S}(\R)$ is a real function such that $f(\pm an^{\alpha}) = \widehat{f}(\pm bn^{\beta}) = 0$ holds for each natural number $n \ge 0.$ Under which conditions on $a,b > 0$ does it holds that $f\equiv 0$? 
\end{question}

This type of questions remains heavily unexplored even in the $\alpha = \beta = \frac{1}{2}$ case, where we believe that a combination of our present techniques with those of \cite{RV} may be useful. 

\section*{Acknowledgments}
We would like to especially thank Felipe Gonçalves and Danylo Radchenko for discussing these problems and ideas about Fourier uniqueness pairs with us. We would also like to thank Christoph Thiele for helpful suggestions during the development of this manuscript. We also thank Emanuel Carneiro and Lucas Oliveira for helpful comments and remarks. Finally, J.P.G.R. acknowledges financial support from the Deutscher Akademischer Austauschdienst (DAAD).

\end{document}